\documentclass[10pt]{article}
\usepackage{amsmath}
\usepackage{amssymb}
\usepackage{amsthm}
\usepackage[compress,noadjust]{cite}
\numberwithin{equation}{section}

\begin{document}
\title{Boundedness of vector-valued intrinsic square functions in Morrey type spaces}
\author{Hua Wang \footnote{E-mail address: wanghua@pku.edu.cn.}\\
\footnotesize{College of Mathematics and Econometrics, Hunan University, Changsha 410082, P. R. China}}
\date{}
\maketitle

\begin{abstract}
In this paper, we will obtain the strong type and weak type estimates for vector-valued analogues of intrinsic square functions in the weighted Morrey spaces $L^{p,\kappa}(w)$ when $1\leq p<\infty$, $0<\kappa<1$, and in the generalized Morrey spaces $L^{p,\Phi}$ for $1\le p<\infty$, where $\Phi$ is a growth function on $(0,\infty)$ satisfying the doubling condition. \\
MSC(2010): 42B25; 42B35 \\
Keywords: Intrinsic square functions; vector-valued inequalities; weighted Morrey spaces; $A_p$ weights; generalized Morrey spaces
\end{abstract}

\section{Introduction and main results}

The intrinsic square functions were first introduced by Wilson in \cite{wilson1,wilson2}; they are defined as follows. For $0<\alpha\le1$, let ${\mathcal C}_\alpha$ be the family of functions $\varphi$ defined on $\mathbb R^n$ such that $\varphi$ has support containing in $\{x\in\mathbb R^n: |x|\le1\}$, $\int_{\mathbb R^n}\varphi(x)\,dx=0$, and for all $x, x'\in \mathbb R^n$,
\begin{equation*}
\big|\varphi(x)-\varphi(x')\big|\leq \big|x-x'\big|^{\alpha}.
\end{equation*}
For $(y,t)\in {\mathbb R}^{n+1}_{+}=\mathbb R^n\times(0,\infty)$ and $f\in L^1_{{loc}}(\mathbb R^n)$, we set
\begin{equation}
A_\alpha(f)(y,t)=\sup_{\varphi\in{\mathcal C}_\alpha}\big|f*\varphi_t(y)\big|=\sup_{\varphi\in{\mathcal C}_\alpha}\bigg|\int_{\mathbb R^n}\varphi_t(y-z)f(z)\,dz\bigg|.
\end{equation}
Then we define the intrinsic square function of $f$ (of order $\alpha$) by the formula
\begin{equation}
\mathcal S_{\alpha}(f)(x)=\left(\iint_{\Gamma(x)}\Big(A_\alpha(f)(y,t)\Big)^2\frac{dydt}{t^{n+1}}\right)^{1/2},
\end{equation}
where $\Gamma(x)$ denotes the usual cone of aperture one:
\begin{equation*}
\Gamma(x)=\big\{(y,t)\in{\mathbb R}^{n+1}_+:|x-y|<t\big\}.
\end{equation*}
Let $\vec{f}=(f_1,f_2,\ldots)$ be a sequence of locally integrable functions on $\mathbb R^n$. For any $x\in\mathbb R^n$, Wilson \cite{wilson2} also defined the vector-valued intrinsic square functions of $\vec{f}$ by
\begin{equation}
\mathcal S_\alpha(\vec{f})(x)=\bigg(\sum_{j=1}^\infty\big|\mathcal S_\alpha(f_j)(x)\big|^2\bigg)^{1/2}.
\end{equation}

In \cite{wilson2}, Wilson has established the following two theorems.

\newtheorem*{thma}{Theorem A}

\begin{thma}[\cite{wilson2}]
Let $0<\alpha\le1$, $1<p<\infty$ and $w\in A_p$(\mbox{Muckenhoupt weight class}). Then there exists a constant $C>0$ independent of $\vec{f}=(f_1,f_2,\ldots)$ such that
\begin{equation*}
\bigg\|\bigg(\sum_{j}\big|\mathcal S_\alpha(f_j)\big|^2\bigg)^{1/2}\bigg\|_{L^p_w}
\leq C \bigg\|\bigg(\sum_{j}\big|f_j\big|^2\bigg)^{1/2}\bigg\|_{L^p_w}.
\end{equation*}
\end{thma}

\newtheorem*{thmb'}{Theorem \=B}

\begin{thmb'}[\cite{wilson2}]
Let $0<\alpha\le1$ and $p=1$. Then for any given weight function $w$ and $\lambda>0$, there exists a constant $C>0$ independent of $\vec{f}=(f_1,f_2,\ldots)$ and $\lambda$ such that
\begin{equation*}
w\bigg(\bigg\{x\in\mathbb R^n:\bigg(\sum_{j}\big|\mathcal S_{\alpha}(f_j)(x)\big|^2\bigg)^{1/2}>\lambda\bigg\}\bigg)
\leq \frac{C}{\lambda}\int_{\mathbb R^n}\bigg(\sum_{j}\big|f_j(x)\big|^2\bigg)^{1/2}Mw(x)\,dx,
\end{equation*}
where $M$ denotes the standard Hardy--Littlewood maximal operator.
\end{thmb'}

If we take $w\in A_1$, then $M(w)(x)\le C\cdot w(x)$ for a.e.$x\in\mathbb R^n$ by the definition of $A_1$ weights (see Section 2). Hence, as a straightforward consequence of Theorem \={B}, we obtain

\newtheorem*{thmb}{Theorem B}

\begin{thmb}
Let $0<\alpha\le1$, $p=1$ and $w\in A_1$. Then there exists a constant $C>0$ independent of $\vec{f}=(f_1,f_2,\ldots)$ such that
\begin{equation*}
\bigg\|\bigg(\sum_{j}\big|\mathcal S_{\alpha}(f_j)\big|^2\bigg)^{1/2}\bigg\|_{WL^1_w}
\leq C \bigg\|\bigg(\sum_{j}\big|f_j\big|^2\bigg)^{1/2}\bigg\|_{L^1_w}.
\end{equation*}
\end{thmb}

In particular, if we take $w$ to be a constant function, then we immediately
get the following

\newtheorem*{thmc}{Theorem C}

\begin{thmc}
Let $0<\alpha\le1$ and $1<p<\infty$. Then there exists a constant $C>0$ independent of $\vec{f}=(f_1,f_2,\ldots)$ such that
\begin{equation*}
\bigg\|\bigg(\sum_{j}\big|\mathcal S_\alpha(f_j)\big|^2\bigg)^{1/2}\bigg\|_{L^p}
\leq C \bigg\|\bigg(\sum_{j}\big|f_j\big|^2\bigg)^{1/2}\bigg\|_{L^p}.
\end{equation*}
\end{thmc}

\newtheorem*{thmd}{Theorem D}

\begin{thmd}
Let $0<\alpha\le1$ and $p=1$. Then there exists a constant $C>0$ independent of $\vec{f}=(f_1,f_2,\ldots)$ such that
\begin{equation*}
\bigg\|\bigg(\sum_{j}\big|\mathcal S_{\alpha}(f_j)\big|^2\bigg)^{1/2}\bigg\|_{WL^1}
\leq C \bigg\|\bigg(\sum_{j}\big|f_j\big|^2\bigg)^{1/2}\bigg\|_{L^1}.
\end{equation*}
\end{thmd}

On the other hand, the classical Morrey spaces $\mathcal L^{p,\lambda}$ were originally introduced by Morrey in \cite{morrey} to study the local behavior of solutions to second order elliptic partial differential equations. Since then, these spaces play an important role in studying the regularity of solutions to partial differential equations. For the boundedness of the Hardy--Littlewood maximal operator, the fractional integral operator and the Calder\'on--Zygmund singular integral operator on these spaces, we refer the reader to \cite{adams,chiarenza,peetre}. In \cite{mizuhara}, Mizuhara introduced the generalized Morrey spaces $L^{p,\Phi}$ which was later extended and studied by many authors (see \cite{guliyev1,guliyev2,guliyev3,lu,nakai}). In \cite{komori}, Komori and Shirai defined the weighted Morrey spaces $L^{p,\kappa}(w)$ which could be
viewed as an extension of weighted Lebesgue spaces, and then discussed the boundedness of the above classical operators in Harmonic Analysis on these weighted spaces. Recently, in \cite{wang1,wang2,wang3}, we have established the strong type and weak type estimates for intrinsic square functions on $L^{p,\Phi}$ and $L^{p,\kappa}(w)$.

For the boundedness of vector-valued intrinsic square functions in the weighted Morrey spaces $L^{p,\kappa}(w)$ for all $1\leq p<\infty$ and $0<\kappa<1$, we will prove

\newtheorem{theorem}{Theorem}[section]

\begin{theorem}\label{mainthm:1}
Let $0<\alpha\le1$, $1<p<\infty$, $0<\kappa<1$ and $w\in A_p$. Then there is a constant $C>0$ independent of $\vec{f}=(f_1,f_2,\ldots)$ such that
\begin{equation*}
\bigg\|\bigg(\sum_{j}\big|\mathcal S_\alpha(f_j)\big|^2\bigg)^{1/2}\bigg\|_{L^{p,\kappa}(w)}
\leq C \bigg\|\bigg(\sum_{j}\big|f_j\big|^2\bigg)^{1/2}\bigg\|_{L^{p,\kappa}(w)}.
\end{equation*}
\end{theorem}

\begin{theorem}\label{mainthm:2}
Let $0<\alpha\le1$, $p=1$, $0<\kappa<1$ and $w\in A_1$. Then there is a constant $C>0$ independent of $\vec{f}=(f_1,f_2,\ldots)$ such that
\begin{equation*}
\bigg\|\bigg(\sum_{j}\big|\mathcal S_{\alpha}(f_j)\big|^2\bigg)^{1/2}\bigg\|_{WL^{1,\kappa}(w)}
\leq C \bigg\|\bigg(\sum_{j}\big|f_j\big|^2\bigg)^{1/2}\bigg\|_{L^{1,\kappa}(w)}.
\end{equation*}
\end{theorem}

For the continuity properties of $\mathcal S_\alpha(\vec{f})$ in $L^{p,\Phi}$ for all $1\leq p<\infty$, we will show that

\begin{theorem}\label{mainthm:3}
Let $0<\alpha\le1$ and $1<p<\infty$. Assume that $\Phi$ satisfies $(\ref{doubling})$ and $1\le D(\Phi)<2^n$, then there is a constant $C>0$ independent of $\vec{f}=(f_1,f_2,\ldots)$ such that
\begin{equation*}
\bigg\|\bigg(\sum_{j}\big|\mathcal S_\alpha(f_j)\big|^2\bigg)^{1/2}\bigg\|_{L^{p,\Phi}}
\leq C\bigg\|\bigg(\sum_{j}\big|f_j\big|^2\bigg)^{1/2}\bigg\|_{L^{p,\Phi}}.
\end{equation*}
\end{theorem}

\begin{theorem}\label{mainthm:4}
Let $0<\alpha\le1$ and $p=1$. Assume that $\Phi$ satisfies $(\ref{doubling})$ and $1\le D(\Phi)<2^n$, then there is a constant $C>0$ independent of $\vec{f}=(f_1,f_2,\ldots)$ such that
\begin{equation*}
\bigg\|\bigg(\sum_{j}\big|\mathcal S_\alpha(f_j)\big|^2\bigg)^{1/2}\bigg\|_{WL^{1,\Phi}}
\leq C\bigg\|\bigg(\sum_{j}\big|f_j\big|^2\bigg)^{1/2}\bigg\|_{L^{1,\Phi}}.
\end{equation*}
\end{theorem}

\section{Notations and definitions}

\newtheorem{defn}[theorem]{Definition}

\subsection{Generalized Morrey spaces}
Let $\Phi=\Phi(r)$, $r>0$, be a growth function, that is, a positive increasing function in $(0,\infty)$ and satisfy the following doubling condition.
\begin{equation}\label{doubling}
\Phi(2r)\le D\cdot\Phi(r), \quad \mbox{for all }\,r>0,
\end{equation}
where $D=D(\Phi)\ge1$ is a doubling constant independent of $r$.

\begin{defn}[\cite{mizuhara}]
Let $1\le p<\infty$. We denote by $L^{p,\Phi}=L^{p,\Phi}(\mathbb R^n)$ the space of all locally integrable functions $f$ defined on $\mathbb R^n$, such that for every $x_0\in\mathbb R^n$ and all $r>0$
\begin{equation}\label{generalized morrey}
\int_{B(x_0,r)}|f(x)|^p\,dx\le C^p\Phi(r),
\end{equation}
where $B(x_0,r)=\{x\in\mathbb R^n:|x-x_0|<r\}$ is the ball centered at $x_0$ and with radius $r>0$. Then we let $\|f\|_{L^{p,\Phi}}$ be the smallest constant $C>0$ satisfying (\ref{generalized morrey}) and $L^{p,\Phi}(\mathbb R^n)$ becomes a Banach space with norm $\|\cdot\|_{L^{p,\Phi}}$.
\end{defn}

Obviously, when $\Phi(r)=r^{\lambda}$ with $0<\lambda<n$, $L^{p,\Phi}$ is just the classical Morrey spaces introduced in \cite{morrey}. We also denote by $WL^{1,\Phi}=WL^{1,\Phi}(\mathbb R^n)$ the generalized weak Morrey spaces of all measurable functions $f$ for which
\begin{equation}\label{general weak morrey}
\sup_{\lambda>0}\lambda\cdot\big|\big\{x\in B(x_0,r):|f(x)|>\lambda\big\}\big|\le C\Phi(r),
\end{equation}
for every $x_0\in\mathbb R^n$ and all $r>0$. The smallest constant $C>0$ satisfying (\ref{general weak morrey}) is also denoted by $\|f\|_{WL^{1,\Phi}}$.

\subsection{Weighted Morrey spaces}

A weight $w$ is a nonnegative, locally integrable function on $\mathbb R^n$, $B=B(x_0,r_B)$ denotes the ball with the center $x_0$ and radius $r_B$. For $1<p<\infty$, a weight function $w$ is said to belong to $A_p$, if there is a constant $C>0$ such that for every ball $B\subseteq \mathbb R^n$,
\begin{equation*}
\left(\frac1{|B|}\int_B w(x)\,dx\right)\left(\frac1{|B|}\int_B w(x)^{-1/{(p-1)}}\,dx\right)^{p-1}\le C.
\end{equation*}
For the case $p=1$, $w\in A_1$, if there is a constant $C>0$ such that for every ball $B\subseteq \mathbb R^n$,
\begin{equation*}
\frac1{|B|}\int_B w(x)\,dx\le C\cdot\underset{x\in B}{\mbox{ess\,inf}}\;w(x).
\end{equation*}
A weight function $w\in A_\infty$ if it satisfies the $A_p$ condition for some $1\leq p<\infty$. It is well known that if $w\in A_p$ with $1\leq p<\infty$, then for any ball $B$, there exists an absolute constant $C>0$ such that
\begin{equation}\label{weights}
w(2B)\le C\,w(B).
\end{equation}
Moreover, if $w\in A_\infty$, then for all balls $B$ and all measurable subsets $E$ of $B$, there exists $\delta>0$ such that
\begin{equation}\label{compare}
\frac{w(E)}{w(B)}\le C\left(\frac{|E|}{|B|}\right)^\delta.
\end{equation}
Given a ball $B$ and $\lambda>0$, $\lambda B$ denotes the ball with the same center as $B$ whose radius is $\lambda$ times that of $B$. For a given weight function $w$ and a measurable set $E$, we also denote the Lebesgue measure of $E$ by $|E|$ and the weighted measure of $E$ by $w(E)$, where $w(E)=\int_E w(x)\,dx$.

Given a weight function $w$ on $\mathbb R^n$, for $1\leq p<\infty$, the weighted Lebesgue space $L^p_w(\mathbb R^n)$ defined as the set of all functions $f$ such that
\begin{equation}
\big\|f\big\|_{L^p_w}=\bigg(\int_{\mathbb R^n}|f(x)|^pw(x)\,dx\bigg)^{1/p}<\infty.
\end{equation}
We also denote by $WL^1_w(\mathbb R^n)$ the weighted weak space consisting of all measurable functions $f$ such that
\begin{equation}
\big\|f\big\|_{WL^1_w}=
\sup_{\lambda>0}\lambda\cdot w\big(\big\{x\in\mathbb R^n:|f(x)|>\lambda \big\}\big)<\infty.
\end{equation}

In particular, for $w$ equals to a constant function, we shall denote $L^p_w(\mathbb R^n)$ and $WL^1_w(\mathbb R^n)$ simply by $L^p(\mathbb R^n)$ and $WL^1(\mathbb R^n)$.

\begin{defn}[\cite{komori}]
Let $1\leq p<\infty$, $0<\kappa<1$ and $w$ be a weight function on $\mathbb R^n$. Then the weighted Morrey space is defined by
\begin{equation*}
L^{p,\kappa}(w)=\big\{f\in L^p_{loc}(w):\big\|f\big\|_{L^{p,\kappa}(w)}<\infty\big\},
\end{equation*}
where
\begin{equation}
\big\|f\big\|_{L^{p,\kappa}(w)}=\sup_B\left(\frac{1}{w(B)^{\kappa}}\int_B|f(x)|^pw(x)\,dx\right)^{1/p}
\end{equation}
and the supremum is taken over all balls $B$ in $\mathbb R^n$.
\end{defn}

For $p=1$ and $0<\kappa<1$, we also denote by $WL^{1,\kappa}(w)$ the weighted weak Morrey spaces of all measurable functions $f$ satisfying
\begin{equation}
\big\|f\big\|_{WL^{1,\kappa}(w)}=\sup_B\sup_{\lambda>0}\frac{1}{w(B)^{\kappa}}\lambda\cdot w\big(\big\{x\in B:|f(x)|>\lambda\big\}\big)<\infty.
\end{equation}

Throughout this paper, the letter $C$ always denote a positive constant independent of the main parameters involved, but it may be different from line to line.

\section{Proofs of Theorems \ref{mainthm:1} and \ref{mainthm:2}}

\begin{proof}[Proof of Theorem $\ref{mainthm:1}$]
Let $\Big(\sum_{j}\big|f_j\big|^2\Big)^{1/2}\in L^{p,\kappa}(w)$ with $1<p<\infty$ and $0<\kappa<1$. Fix a ball $B=B(x_0,r_B)\subseteq\mathbb R^n$ and decompose $f_j=f^0_j+f^\infty_j$, where $f^0_j=f_j\cdot\chi_{_{2B}}$ and $\chi_{_{2B}}$ denotes the characteristic function of $2B=B(x_0,2r_B)$, $j=1,2,\ldots$. Then we write
\begin{equation*}
\begin{split}
&\frac{1}{w(B)^{\kappa/p}}\Bigg(\int_B\bigg(\sum_{j}\big|\mathcal S_\alpha(f_j)(x)\big|^2\bigg)^{p/2}w(x)\,dx\Bigg)^{1/p}\\
\leq\ &\frac{1}{w(B)^{\kappa/p}}\Bigg(\int_B\bigg(\sum_{j}\big|\mathcal S_\alpha(f^0_j)(x)\big|^2\bigg)^{p/2}w(x)\,dx\Bigg)^{1/p}\\
& +\frac{1}{w(B)^{\kappa/p}}\Bigg(\int_B\bigg(\sum_{j}\big|\mathcal S_\alpha(f^\infty_j)(x)\big|^2\bigg)^{p/2}w(x)\,dx\Bigg)^{1/p}\\
=\ &I_1+I_2.
\end{split}
\end{equation*}
Using Theorem A and the inequality (\ref{weights}), we have
\begin{equation*}
\begin{split}
I_1&\leq \frac{1}{w(B)^{\kappa/p}}
\bigg\|\bigg(\sum_{j}\big|\mathcal S_\alpha(f^0_j)\big|^2\bigg)^{1/2}\bigg\|_{L^p_w}\\
&\leq C\cdot\frac{1}{w(B)^{\kappa/p}}
\Bigg(\int_{2B}\bigg(\sum_{j}\big|f_j(x)\big|^2\bigg)^{p/2}w(x)\,dx\Bigg)^{1/p}\\
&\leq C\bigg\|\bigg(\sum_{j}\big|f_j\big|^2\bigg)^{1/2}\bigg\|_{L^{p,\kappa}(w)}
\cdot\frac{w(2B)^{\kappa/p}}{w(B)^{\kappa/p}}\\
&\leq C\bigg\|\bigg(\sum_{j}\big|f_j\big|^2\bigg)^{1/2}\bigg\|_{L^{p,\kappa}(w)}.
\end{split}
\end{equation*}
Let us now turn to estimate the other term $I_2$. For any $\varphi\in{\mathcal C}_\alpha$, $0<\alpha\le1$, $j=1,2,\ldots$ and $(y,t)\in\Gamma(x)$, we have
\begin{align}\label{ineq:5.1}
\big|f^\infty_j*\varphi_t(y)\big|&=\bigg|\int_{(2B)^c}\varphi_t(y-z)f_j(z)\,dz\bigg|\notag\\
&\leq C\cdot t^{-n}\int_{(2B)^c\cap\{z:|y-z|\le t\}}\big|f_j(z)\big|\,dz\notag\\
&\leq C\cdot t^{-n}\sum_{\ell=1}^\infty\int_{(2^{\ell+1}B\backslash 2^{\ell}B)\cap\{z:|y-z|\le t\}}\big|f_j(z)\big|\,dz.
\end{align}
For any $x\in B$, $(y,t)\in\Gamma(x)$ and $z\in\big(2^{\ell+1}B\backslash 2^{\ell}B\big)\cap B(y,t)$, then by a direct computation, we can easily see that
\begin{equation}\label{ineq:5.2}
2t\geq |x-y|+|y-z|\geq |x-z|\geq |z-x_0|-|x-x_0|\geq 2^{\ell-1}r_B.
\end{equation}
Thus, by using the above inequalities (\ref{ineq:5.1}) and (\ref{ineq:5.2}) together with Minkowski's inequality for integrals, we deduce
\begin{equation*}
\begin{split}
\big|\mathcal S_{\alpha}(f^{\infty}_j)(x)\big|&=\left(\iint_{\Gamma(x)}\sup_{\varphi\in{\mathcal C}_\alpha}\big|f^{\infty}_j*\varphi_t(y)\big|^2\frac{dydt}{t^{n+1}}\right)^{1/2}\\
&\leq C\left(\int_{2^{\ell-2}r_B}^\infty\int_{|x-y|<t}\bigg|t^{-n}\sum_{\ell=1}^\infty\int_{2^{\ell+1}B\backslash 2^{\ell}B}\big|f_j(z)\big|\,dz\bigg|^2\frac{dydt}{t^{n+1}}\right)^{1/2}\\
&\le C\left(\sum_{\ell=1}^\infty\int_{2^{\ell+1}B\backslash 2^{\ell}B}\big|f_j(z)\big|\,dz\right)
\left(\int_{2^{\ell-2}r_B}^\infty\frac{dt}{t^{2n+1}}\right)^{1/2}\\
&\leq C\sum_{\ell=1}^\infty\frac{1}{|2^{\ell+1}B|}\int_{2^{\ell+1}B\backslash 2^{\ell}B}\big|f_j(z)\big|\,dz.
\end{split}
\end{equation*}
Then by duality and Cauchy--Schwarz inequality, we get
\begin{align}\label{key estimate}
\bigg(\sum_{j}\big|\mathcal S_\alpha(f^{\infty}_j)(x)\big|^2\bigg)^{1/2}&\leq
C\Bigg(\sum_{j}\bigg|\sum_{\ell=1}^\infty\frac{1}{|2^{\ell+1}B|}\int_{2^{\ell+1}B\backslash 2^{\ell}B}\big|f_j(z)\big|\,dz\bigg|^2\Bigg)^{1/2}\notag\\
&\leq C\sup_{(\sum_j|\zeta_j|^2)^{1/2}\leq1}\sum_{j}\bigg(\sum_{\ell=1}^\infty\frac{1}{|2^{\ell+1}B|}\int_{2^{\ell+1}B}\big|f_j(z)\big|\,dz\cdot\zeta_j\bigg)\notag\\
&\leq C\sum_{\ell=1}^\infty\frac{1}{|2^{\ell+1}B|}\int_{2^{\ell+1}B}\sup_{(\sum_j|\zeta_j|^2)^{1/2}\leq1}\bigg(\sum_{j}\big|f_j(z)\big|\cdot\zeta_j\bigg)\,dz\notag\\
&\leq C\sum_{\ell=1}^\infty\frac{1}{|2^{\ell+1}B|}\int_{2^{\ell+1}B}\bigg(\sum_{j}\big|f_j(z)\big|^2\bigg)^{1/2}\,dz.
\end{align}
Furthermore, it follows from H\"older's inequality, (\ref{key estimate}) and the $A_p$ condition that
\begin{equation*}
\begin{split}
\bigg(\sum_{j}\big|\mathcal S_\alpha(f^{\infty}_j)(x)\big|^2\bigg)^{1/2}&\leq C\sum_{\ell=1}^\infty\frac{1}{|2^{\ell+1}B|}\Bigg(\int_{2^{\ell+1}B}
\bigg(\sum_{j}\big|f_j(z)\big|^2\bigg)^{p/2}w(z)\,dz\Bigg)^{1/p}\\
&\ \times\left(\int_{2^{\ell+1}B}w(z)^{-{p'}/p}\,dz\right)^{1/{p'}}\\
&\leq C\bigg\|\bigg(\sum_{j}\big|f_j\big|^2\bigg)^{1/2}\bigg\|_{L^{p,\kappa}(w)}
\cdot w\big(2^{\ell+1}B\big)^{{(\kappa-1)}/p},
\end{split}
\end{equation*}
where we denote the conjugate exponent of $p>1$ by $p'=p/{(p-1)}$. Since $w\in A_p\subset A_\infty$ for all $1<p<\infty$. Hence, we apply the inequality (\ref{compare}) to obtain
\begin{equation*}
\begin{split}
I_2&\leq C\bigg\|\bigg(\sum_{j}\big|f_j\big|^2\bigg)^{1/2}\bigg\|_{L^{p,\kappa}(w)}
\cdot\sum_{\ell=1}^\infty\frac{w(B)^{{(1-\kappa)}/p}}{w(2^{\ell+1}B)^{{(1-\kappa)}/p}}\\
\end{split}
\end{equation*}
\begin{equation*}
\begin{split}
&\leq C\bigg\|\bigg(\sum_{j}\big|f_j\big|^2\bigg)^{1/2}\bigg\|_{L^{p,\kappa}(w)}
\cdot\sum_{\ell=1}^\infty\left(\frac{|B|}{|2^{\ell+1}B|}\right)^{\delta\cdot{(1-\kappa)}/p}\\
&\leq C\bigg\|\bigg(\sum_{j}\big|f_j\big|^2\bigg)^{1/2}\bigg\|_{L^{p,\kappa}(w)},
\end{split}
\end{equation*}
where the last series is convergent since $0<\kappa<1$ and $\delta>0$. Summarizing the above two estimates for $I_1$ and $I_2$, and then taking the supremum over all balls $B\subseteq\mathbb R^n$, we complete the proof of Theorem \ref{mainthm:1}.
\end{proof}

\begin{proof}[Proof of Theorem $\ref{mainthm:2}$]
Let $\Big(\sum_{j}\big|f_j\big|^2\Big)^{1/2}\in L^{1,\kappa}(w)$ with $0<\kappa<1$. Fix a ball $B=B(x_0,r_B)\subseteq\mathbb R^n$, we set $f_j=f^0_j+f^\infty_j$, where $f^0_j=f_j\cdot\chi_{_{2B}}$, $j=1,2,\ldots.$ Then for any given $\lambda>0$, one writes
\begin{equation*}
\begin{split}
& w\bigg(\bigg\{x\in B:\bigg(\sum_{j}\big|\mathcal S_{\alpha}(f_j)(x)\big|^2\bigg)^{1/2}>\lambda\bigg\}\bigg)\\
\leq\ &w\bigg(\bigg\{x\in B:\bigg(\sum_{j}\big|\mathcal S_{\alpha}(f^0_j)(x)\big|^2\bigg)^{1/2}>\lambda/2\bigg\}\bigg)\\
& +w\bigg(\bigg\{x\in B:\bigg(\sum_{j}\big|\mathcal S_{\alpha}(f^\infty_j)(x)\big|^2\bigg)^{1/2}>\lambda/2\bigg\}\bigg)\\
=\ &I'_1+I'_2.
\end{split}
\end{equation*}
Theorem B and the inequality (\ref{weights}) imply that
\begin{equation*}
\begin{split}
I'_1&\leq \frac{2}{\lambda}\cdot
\bigg\|\bigg(\sum_{j}\big|\mathcal S_\alpha(f^0_j)\big|^2\bigg)^{1/2}\bigg\|_{WL^1_w}\\
&\leq \frac{C}{\lambda}\cdot
\Bigg(\int_{2B}\bigg(\sum_{j}\big|f_j(x)\big|^2\bigg)^{1/2}w(x)\,dx\Bigg)\\
&\leq \frac{C\cdot w(2B)^\kappa}{\lambda}
\bigg\|\bigg(\sum_{j}\big|f_j\big|^2\bigg)^{1/2}\bigg\|_{L^{1,\kappa}(w)}\\
&\leq \frac{C\cdot w(B)^\kappa}{\lambda}
\bigg\|\bigg(\sum_{j}\big|f_j\big|^2\bigg)^{1/2}\bigg\|_{L^{1,\kappa}(w)}.
\end{split}
\end{equation*}
We now turn to deal with the other term $I'_2$. In the proof of Theorem \ref{mainthm:1}, we have already showed that for any $x\in B$ (see (\ref{key estimate})),
\begin{equation*}
\bigg(\sum_{j}\big|\mathcal S_\alpha(f^{\infty}_j)(x)\big|^2\bigg)^{1/2}\leq C\sum_{\ell=1}^\infty\frac{1}{|2^{\ell+1}B|}\int_{2^{\ell+1}B}\bigg(\sum_{j}\big|f_j(z)\big|^2\bigg)^{1/2}\,dz.
\end{equation*}
It follows directly from the $A_1$ condition that
\begin{equation*}
\begin{split}
\bigg(\sum_{j}\big|\mathcal S_\alpha(f^{\infty}_j)(x)\big|^2\bigg)^{1/2}
&\leq C\sum_{\ell=1}^\infty\frac{\underset{z\in 2^{\ell+1}B}{\mbox{ess\,inf}}\,w(z)}{w(2^{\ell+1}B)}
\int_{2^{\ell+1}B}\bigg(\sum_{j}\big|f_j(z)\big|^2\bigg)^{1/2}\,dz\\
&\leq C\sum_{\ell=1}^\infty\frac{1}{w(2^{\ell+1}B)}
\int_{2^{\ell+1}B}\bigg(\sum_{j}\big|f_j(z)\big|^2\bigg)^{1/2}w(z)\,dz\\
&\leq C\bigg\|\bigg(\sum_{j}\big|f_j\big|^2\bigg)^{1/2}\bigg\|_{L^{1,\kappa}(w)}
\sum_{\ell=1}^\infty\frac{1}{w(2^{\ell+1}B)^{1-\kappa}}.
\end{split}
\end{equation*}
In addition, since $w\in A_1\subset A_\infty$, then by the inequality (\ref{compare}), we can see that for all $x\in B$,
\begin{align}\label{pointwise}
\bigg(\sum_{j}\big|\mathcal S_\alpha(f^{\infty}_j)(x)\big|^2\bigg)^{1/2}&\leq
C\bigg\|\bigg(\sum_{j}\big|f_j\big|^2\bigg)^{1/2}\bigg\|_{L^{1,\kappa}(w)}\cdot
\frac{1}{w(B)^{1-\kappa}}\sum_{\ell=1}^\infty\frac{w(B)^{1-\kappa}}{w(2^{\ell+1}B)^{1-\kappa}}\notag\\
&\leq C\bigg\|\bigg(\sum_{j}\big|f_j\big|^2\bigg)^{1/2}\bigg\|_{L^{1,\kappa}(w)}\cdot
\frac{1}{w(B)^{1-\kappa}}\sum_{\ell=1}^\infty\left(\frac{|B|}{|2^{\ell+1}B|}\right)^{\delta^*\cdot(1-\kappa)}\notag\\
&\leq C\bigg\|\bigg(\sum_{j}\big|f_j\big|^2\bigg)^{1/2}\bigg\|_{L^{1,\kappa}(w)}\cdot
\frac{1}{w(B)^{1-\kappa}},
\end{align}
where in the last inequality we have used the fact that $\delta^*\cdot(1-\kappa)>0$. If $\bigg\{x\in B:\bigg(\sum_{j}\big|\mathcal S_{\alpha}(f^\infty_j)(x)\big|^2\bigg)^{1/2}>\lambda/2\bigg\}=\emptyset$, then the inequality
\begin{equation*}
I'_2\leq \frac{C\cdot w(B)^\kappa}{\lambda}
\bigg\|\bigg(\sum_{j}\big|f_j\big|^2\bigg)^{1/2}\bigg\|_{L^{1,\kappa}(w)}
\end{equation*}
holds trivially. Now if instead we suppose that
$$\bigg\{x\in B:\bigg(\sum_{j}\big|\mathcal S_{\alpha}(f^\infty_j)(x)\big|^2\bigg)^{1/2}>\lambda/2\bigg\}\neq\emptyset.$$
Then by the pointwise inequality (\ref{pointwise}), we have
\begin{equation*}
\lambda\leq C\bigg\|\bigg(\sum_{j}\big|f_j\big|^2\bigg)^{1/2}\bigg\|_{L^{1,\kappa}(w)}\cdot
\frac{1}{w(B)^{1-\kappa}},
\end{equation*}
which is equivalent to
\begin{equation*}
w(B)\leq\frac{C\cdot w(B)^\kappa}{\lambda}\bigg\|\bigg(\sum_{j}\big|f_j\big|^2\bigg)^{1/2}\bigg\|_{L^{1,\kappa}(w)}.
\end{equation*}
Therefore
\begin{equation*}
I'_2\leq w(B)\leq\frac{C\cdot w(B)^\kappa}{\lambda}\bigg\|\bigg(\sum_{j}\big|f_j\big|^2\bigg)^{1/2}\bigg\|_{L^{1,\kappa}(w)}.
\end{equation*}
Summing up the above estimates for $I'_1$ and $I'_2$, and then taking the supremum over all balls $B\subseteq\mathbb R^n$ and all $\lambda>0$, we finish the proof of Theorem \ref{mainthm:2}.
\end{proof}

\section{Proofs of Theorems \ref{mainthm:3} and \ref{mainthm:4}}

\begin{proof}[Proof of Theorem $\ref{mainthm:3}$]
Let $\Big(\sum_{j}\big|f_j\big|^2\Big)^{1/2}\in L^{p,\Phi}$ with $1<p<\infty$. For any ball $B=B(x_0,r)\subseteq\mathbb R^n$ with $x_0\in\mathbb R^n$ and $r>0$, we write $f_j=f^0_j+f^\infty_j$, where $f^0_j=f_j\cdot\chi_{_{2B}}$, $j=1,2,\ldots$. Then we have
\begin{equation*}
\begin{split}
&\frac{1}{\Phi(r)^{1/p}}\Bigg(\int_{B(x_0,r)}\bigg(\sum_{j}\big|\mathcal S_\alpha(f_j)(x)\big|^2\bigg)^{p/2}\,dx\Bigg)^{1/p}\\
\leq\ &\frac{1}{\Phi(r)^{1/p}}\Bigg(\int_{B(x_0,r)}\bigg(\sum_{j}\big|\mathcal S_\alpha(f^0_j)(x)\big|^2\bigg)^{p/2}\,dx\Bigg)^{1/p}\\
& +\frac{1}{\Phi(r)^{1/p}}\Bigg(\int_{B(x_0,r)}\bigg(\sum_{j}\big|\mathcal S_\alpha(f^{\infty}_j)(x)\big|^2\bigg)^{p/2}\,dx\Bigg)^{1/p}\\
=\ &J_1+J_2.
\end{split}
\end{equation*}
Applying Theorem C and the doubling condition (\ref{doubling}), we obtain
\begin{equation*}
\begin{split}
J_1&\leq \frac{1}{\Phi(r)^{1/p}}
\bigg\|\bigg(\sum_{j}\big|\mathcal S_\alpha(f^0_j)\big|^2\bigg)^{1/2}\bigg\|_{L^p}\\
&\leq C\cdot\frac{1}{\Phi(r)^{1/p}}
\Bigg(\int_{2B}\bigg(\sum_{j}\big|f_j(x)\big|^2\bigg)^{p/2}\,dx\Bigg)^{1/p}\\
&\leq C\bigg\|\bigg(\sum_{j}\big|f_j\big|^2\bigg)^{1/2}\bigg\|_{L^{p,\Phi}}
\cdot\frac{\Phi(2r)^{1/p}}{\Phi(r)^{1/p}}\\
&\leq C\bigg\|\bigg(\sum_{j}\big|f_j\big|^2\bigg)^{1/2}\bigg\|_{L^{p,\Phi}}.
\end{split}
\end{equation*}
We now turn to estimate the other term $J_2$. We first use the inequality (\ref{key estimate}) and H\"older's inequality to obtain
\begin{equation*}
\begin{split}
\bigg(\sum_{j}\big|\mathcal S_\alpha(f^{\infty}_j)(x)\big|^2\bigg)^{1/2}
&\leq C\sum_{\ell=1}^\infty\frac{1}{|B(x_0,2^{\ell+1}r)|}\int_{2^{\ell+1}B}
\bigg(\sum_{j}\big|f_j(z)\big|^2\bigg)^{1/2}\,dz\\
&\leq C\sum_{\ell=1}^\infty\frac{1}{|B(x_0,2^{\ell+1}r)|^{1/p}}\Bigg(\int_{2^{\ell+1}B}
\bigg(\sum_{j}\big|f_j(z)\big|^2\bigg)^{p/2}\,dz\Bigg)^{1/p}\\
&\leq C\bigg\|\bigg(\sum_{j}\big|f_j\big|^2\bigg)^{1/2}\bigg\|_{L^{p,\Phi}}
\cdot\sum_{\ell=1}^\infty\frac{\Phi(2^{\ell+1}r)^{1/p}}{|B(x_0,2^{\ell+1}r)|^{1/p}}.
\end{split}
\end{equation*}
Hence
\begin{equation*}
J_2\leq C\bigg\|\bigg(\sum_{j}\big|f_j\big|^2\bigg)^{1/2}\bigg\|_{L^{p,\Phi}}
\cdot\sum_{\ell=1}^\infty\frac{|B(x_0,r)|^{1/p}}{\Phi(r)^{1/p}}
\cdot\frac{\Phi(2^{\ell+1}r)^{1/p}}{|B(x_0,2^{\ell+1}r)|^{1/p}}.
\end{equation*}
Since $1\le D(\Phi)<2^n$, then by using the doubling condition (\ref{doubling}) of $\Phi$, we know that
\begin{align}\label{C}
\sum_{\ell=1}^\infty\frac{|B(x_0,r)|^{1/p}}{\Phi(r)^{1/p}}
\cdot\frac{\Phi(2^{\ell+1}r)^{1/p}}{|B(x_0,2^{\ell+1}r)|^{1/p}}
&\leq C\sum_{\ell=1}^\infty\left(\frac{D(\Phi)}{2^{n}}\right)^{{(\ell+1)}/p}\notag\\
&\leq C.
\end{align}
Therefore
\begin{equation*}
J_2\leq C\bigg\|\bigg(\sum_{j}\big|f_j\big|^2\bigg)^{1/2}\bigg\|_{L^{p,\Phi}}.
\end{equation*}
Combining the above estimates for $J_1$ and $J_2$, and then taking the supremum over all balls $B=B(x_0,r)\subseteq\mathbb R^n$, we complete the proof of Theorem \ref{mainthm:3}.
\end{proof}

\begin{proof}[Proof of Theorem $\ref{mainthm:4}$]
Let $\Big(\sum_{j}\big|f_j\big|^2\Big)^{1/2}\in L^{1,\Phi}$. For each fixed ball $B=B(x_0,r)\subseteq\mathbb R^n$, we again decompose $f_j$ as $f_j=f^0_j+f^\infty_j$, where $f^0_j=f_j\cdot\chi_{_{2B}}$, $j=1,2,\ldots$. For any given $\lambda>0$, then we write
\begin{equation*}
\begin{split}
& \bigg|\bigg\{x\in B:\bigg(\sum_{j}\big|\mathcal S_{\alpha}(f_j)(x)\big|^2\bigg)^{1/2}>\lambda\bigg\}\bigg|\\
\leq\ &\bigg|\bigg\{x\in B:\bigg(\sum_{j}\big|\mathcal S_{\alpha}(f^0_j)(x)\big|^2\bigg)^{1/2}>\lambda/2\bigg\}\bigg|\\
& +\bigg|\bigg\{x\in B:\bigg(\sum_{j}\big|\mathcal S_{\alpha}(f^\infty_j)(x)\big|^2\bigg)^{1/2}>\lambda/2\bigg\}\bigg|\\
=\ &J'_1+J'_2.
\end{split}
\end{equation*}
Theorem D and the doubling condition (\ref{doubling}) imply that
\begin{equation*}
\begin{split}
J'_1&\leq \frac{2}{\lambda}\cdot
\bigg\|\bigg(\sum_{j}\big|\mathcal S_\alpha(f^0_j)\big|^2\bigg)^{1/2}\bigg\|_{WL^1}\\
&\leq \frac{C}{\lambda}\cdot
\Bigg(\int_{2B}\bigg(\sum_{j}\big|f_j(x)\big|^2\bigg)^{1/2}\,dx\Bigg)\\
&\leq \frac{C\cdot\Phi(2r)}{\lambda}
\bigg\|\bigg(\sum_{j}\big|f_j\big|^2\bigg)^{1/2}\bigg\|_{L^{1,\Phi}}\\
&\leq \frac{C\cdot\Phi(r)}{\lambda}
\bigg\|\bigg(\sum_{j}\big|f_j\big|^2\bigg)^{1/2}\bigg\|_{L^{1,\Phi}}.
\end{split}
\end{equation*}
We turn our attention to the estimate of $J'_2$. Using the preceding estimate (\ref{key estimate}), we can deduce that for all $x\in B(x_0,r)$,
\begin{equation*}
\begin{split}
\bigg(\sum_{j}\big|\mathcal S_\alpha(f^{\infty}_j)(x)\big|^2\bigg)^{1/2}&\leq C\sum_{\ell=1}^\infty\frac{1}{|B(x_0,2^{\ell+1}r)|}\int_{2^{\ell+1}B}
\bigg(\sum_{j}\big|f_j(z)\big|^2\bigg)^{1/2}\,dz\\
&\leq C\bigg\|\bigg(\sum_{j}\big|f_j\big|^2\bigg)^{1/2}\bigg\|_{L^{1,\Phi}}
\sum_{\ell=1}^\infty\frac{\Phi(2^{\ell+1}r)}{|B(x_0,2^{\ell+1}r)|}\\
&\leq C\bigg\|\bigg(\sum_{j}\big|f_j\big|^2\bigg)^{1/2}\bigg\|_{L^{1,\Phi}}
\cdot\frac{\Phi(r)}{|B(x_0,r)|}\\
&\ \times\sum_{\ell=1}^\infty
\frac{|B(x_0,r)|}{\Phi(r)}\cdot\frac{\Phi(2^{\ell+1}r)}{|B(x_0,2^{\ell+1}r)|}.
\end{split}
\end{equation*}
Note that $1\le D(\Phi)<2^n$. Arguing as in the proof of (\ref{C}), we can get
\begin{align}
\sum_{\ell=1}^\infty\frac{|B(x_0,r)|}{\Phi(r)}\cdot\frac{\Phi(2^{\ell+1}r)}{|B(x_0,2^{\ell+1}r)|}&\le
\sum_{\ell=1}^\infty\left(\frac{D(\Phi)}{2^n}\right)^{\ell+1}\notag\\
&\le C.
\end{align}
Hence, for any $x\in B(x_0,r)$,
\begin{equation}\label{pointwise2}
\bigg(\sum_{j}\big|\mathcal S_\alpha(f^{\infty}_j)(x)\big|^2\bigg)^{1/2}
\leq C\bigg\|\bigg(\sum_{j}\big|f_j\big|^2\bigg)^{1/2}\bigg\|_{L^{1,\Phi}}
\cdot\frac{\Phi(r)}{|B(x_0,r)|}.
\end{equation}
If $\bigg\{x\in B:\bigg(\sum_{j}\big|\mathcal S_{\alpha}(f^\infty_j)(x)\big|^2\bigg)^{1/2}>\lambda/2\bigg\}=\emptyset$, then the inequality
\begin{equation*}
J'_2\leq\frac{C\cdot\Phi(r)}{\lambda}\bigg\|\bigg(\sum_{j}\big|f_j\big|^2\bigg)^{1/2}\bigg\|_{L^{1,\Phi}}
\end{equation*}
holds trivially. Now we may suppose that $$\bigg\{x\in B:\bigg(\sum_{j}\big|\mathcal S_{\alpha}(f^\infty_j)(x)\big|^2\bigg)^{1/2}>\lambda/2\bigg\}\neq\emptyset.$$
Then by the pointwise inequality (\ref{pointwise2}), we can see that
\begin{equation*}
\lambda\leq C\bigg\|\bigg(\sum_{j}\big|f_j\big|^2\bigg)^{1/2}\bigg\|_{L^{1,\Phi}}
\cdot\frac{\Phi(r)}{|B(x_0,r)|},
\end{equation*}
which is equivalent to
\begin{equation*}
\big|B(x_0,r)\big|\leq \frac{C\cdot\Phi(r)}{\lambda}\bigg\|\bigg(\sum_{j}\big|f_j\big|^2\bigg)^{1/2}\bigg\|_{L^{1,\Phi}}.
\end{equation*}
Therefore
\begin{equation*}
J'_2\leq\big|B(x_0,r)\big|\leq \frac{C\cdot\Phi(r)}{\lambda}
\bigg\|\bigg(\sum_{j}\big|f_j\big|^2\bigg)^{1/2}\bigg\|_{L^{1,\Phi}}.
\end{equation*}
Summing up the above estimates for $J'_1$ and $J'_2$, and then taking the supremum over all balls $B=B(x_0,r)\subseteq\mathbb R^n$ and all $\lambda>0$, we conclude the proof of Theorem \ref{mainthm:4}.
\end{proof}


\begin{thebibliography}{99}

\bibitem{adams} D. R. Adams, A note on Riesz potentials, Duke Math. J, \textbf{42}(1975), 765--778.
\bibitem{chiarenza} F. Chiarenza and M. Frasca, Morrey spaces and Hardy--Littlewood maximal function, Rend. Math. Appl, \textbf{7}(1987), 273--279.
\bibitem{garcia} J. Garcia-Cuerva and J. L. Rubio de Francia, Weighted Norm Inequalities and Related Topics, North-Holland, Amsterdam, 1985.
\bibitem{guliyev1} V. S. Guliyev, Boundedness of the maximal, potential and singular operators in the generalized Morrey spaces, J. Inequal. Appl., Article ID 503948, (2009).
\bibitem{guliyev2} V. S. Guliyev, S. S. Aliyev and T. Karaman, Boundedness of a class of sublinear operators and their commutators on generalized Morrey spaces, Abstr. Appl. Anal., Article ID 356041, (2011).
\bibitem{guliyev3} V. S. Guliyev, S. S. Aliyev, T. Karaman and P. S. Shukurov, Boundedness of sublinear operators and commutators on generalized Morrey spaces, Integr. Equ. Oper. Theory, \textbf{71}(2011), 327--355.
\bibitem{komori} Y. Komori and S. Shirai, Weighted Morrey spaces and a singular integral operator, Math. Nachr, \textbf{282}(2009), 219--231.
\bibitem{lu} S. Z. Lu, D. C. Yang and Z. S. Zhou, Sublinear operators with rough kernel on generalized Morrey spaces, Hokkaido Math. J, \textbf{27}(1998), 219--232.
\bibitem{mizuhara} T. Mizuhara, Boundedness of some classical operators on generalized Morrey spaces, Harmonic Analysis, ICM-90 Satellite Conference Proceedings, Springer-Verlag, Tokyo, (1991), 183--189.
\bibitem{morrey} C. B. Morrey, On the solutions of quasi-linear elliptic partial differential equations, Trans. Amer. Math. Soc, \textbf{43}(1938), 126--166.
\bibitem{nakai}  E. Nakai, Hardy--Littlewood maximal operator, singular integral operators and Riesz potentials on generalized Morrey spaces, Math. Nachr., \textbf{166}(1994), 95--103.
\bibitem{muckenhoupt} B. Muckenhoupt, Weighted norm inequalities for the Hardy maximal function, Trans. Amer. Math. Soc, \textbf{165}(1972), 207--226.
\bibitem{peetre} J. Peetre, On the theory of $\mathcal L_{p,\lambda}$ spaces, J. Funct. Anal, \textbf{4}(1969),
    71--87.
\bibitem{wang1} H. Wang, Intrinsic square functions on the weighted Morrey spaces, J. Math. Anal. Appl, \textbf{396}(2012), 302--314.
\bibitem{wang2} H. Wang, Weak type estimates for intrinsic square functions on weighted Morrey spaces, Anal. Theory Appl, \textbf{29}(2013),  104--119.
\bibitem{wang3} H. Wang, Boundedness of intrinsic square functions on generalized Morrey spaces, Georgian Math. J, to appear.
\bibitem{wilson1} M. Wilson, The intrinsic square function, Rev. Mat. Iberoamericana, \textbf{23}(2007), 771--791.
\bibitem{wilson2} M. Wilson, Weighted Littlewood--Paley Theory and Exponential-Square Integrability, Lecture Notes in Math, Vol 1924, Springer-Verlag, 2007.

\end{thebibliography}
\end{document}